\documentclass[11pt,oneside,reqno]{amsart}

\usepackage{yhmath,ytableau, mathdots, relsize, textcomp, amsthm, mathtools}
\usepackage{underoverlap}
\usepackage{enumerate}
\numberwithin{equation}{section}
\theoremstyle{definition}
\newtheorem{Definition}{Definition}[section]
\newtheorem{Example}[Definition]{Example}
\newtheorem{Remark}[Definition]{Remark}
\theoremstyle{plain}
\newtheorem{Theorem}[Definition]{Theorem}

\newtheorem{Lemma}[Definition]{Lemma}

\usepackage{amssymb}        
\usepackage{mathrsfs}       
\usepackage{eucal}          
\usepackage{stmaryrd}       

\usepackage{cancel}         
\usepackage[english]{babel}
\usepackage[T2A]{fontenc}

\usepackage{geometry}
\usepackage{url}
\usepackage[utf8]{inputenc}
\usepackage{hyperref}

\usepackage{tikz}
\usetikzlibrary{arrows, matrix}
\usepackage{tikz-cd}
\usepackage{subcaption}         
\usepackage{arydshln}       

\newcommand{\al}{\alpha}

\newcommand{\ep}{\varepsilon}

\newcommand{\la}{\lambda}
\newcommand{\La}{\Lambda}

\newcommand{\Z}{\mathbb{Z}}

\newcommand{\Be}{\boldsymbol{e}}

\newcommand{\op}{\operatorname}

\DeclareMathOperator{\Hom}{Hom}

\renewcommand{\tilde}{\widetilde}
\newcommand{\YD}{\mathrm{YD}}
\newcommand{\SSYT}{\mathrm{SSYT}}
\newcommand{\GTP}{\mathrm{GTP}} 
\newcommand{\wt}{\mathrm{wt}}
\newcommand{\tf}{\tilde{f}}
\newcommand{\te}{\tilde{e}}

\title{Gelfand-Tsetlin Crystals}
\author{Jonas T. Hartwig \and O'Neill Kingston}
\date{\today}

\address{Department of Mathematics, Iowa State University, Ames, IA-50011, USA}
\email{jth@iastate.edu}
\urladdr{http://jthartwig.net}

\address{Department of Mathematics, Iowa State University, Ames, IA-50011, USA}
\email{oneillk@iastate.edu}
\urladdr{https://sites.google.com/iastate.edu/oneill-kingston/home}

\begin{document}

\maketitle

\begin{abstract}
We give a crystal structure on the set of Gelfand-Tsetlin patterns which parametrize bases for finite-dimensional irreducible representations of the general linear Lie algebra. The crystal data are given in closed form, expressed using tropical polynomial functions of the entries of the patterns. We prove that with this crystal structure, the natural bijection between Gelfand-Tsetlin patterns and semistandard Young tableaux is a crystal isomorphism.
\end{abstract}


\ytableausetup{centertableaux}  

\section{Introduction}

The introduction of crystal bases in the 1990's by Kashiwara (\cite{Kashiwara1990}, \cite{Kashiwara1991}, \cite{KN1994}) and Lusztig \cite{lusztig} was a breakthrough in the representation theory of Lie algebras and quantum groups, structures that have become ubiquitous in modern physics, algebra and geometry.  
A crystal basis can be combinatorially identified with certain directed graphs whose edges are labeled by simple root vectors.
 At the same time, the vertex set of the graph is a basis for a highest-weight representation of a quantum group (as described, for instance, by Hong and Kang in \cite{hongkang}), and its combinatorial structure is naturally compatible with taking tensor products, describing branching rules, and much more.  Crystal bases for all classical Lie algebras can be realized in terms of Kashiwara-Nakashima tableaux. The semistandard Young tableaux that they generalize have been used extensively in representation theory over the last century.

Another long-running undercurrent to this area of study comes from a more purely combinatorial perspective.  When the utility of applying Young tableaux to problems in representation theory became clear, many more generalizations were made than those discussed above.  The era of computers accelerated the growth of interest in this area, and today there are robust communities of mathematicians whose work is focused on coding efficient representations (in the colloquial sense) of these structures, often in Python and Sage, so that these may then be used to attack problems in algebra, combinatorics, geometry and beyond.  With a high level of research output surrounding tableaux and tableau-like structures, an active sub-discipline is the effort to identify when two seemingly-different types of structure are in fact equivalent.  As described by Sheats in \cite{sheats}, these enumeration problems can frequently be difficult, but they can also illuminate surprising connections between areas of mathematics that appeared to have little in common. 

It is well-known that semistandard Young tableaux (SSYT) are in bijection with Gelfand-Tsetlin patterns (GTPs) (\cite{geltse1950}), arrays of numbers that were introduced using branching rules for the general linear Lie algebra and which have found many subsequent applications. 
  We therefore have two different combinatorial descriptions of the same algebraic objects, one in terms of tableaux and the other in terms of patterns.

Any set in bijection with the vertices of a crystal graph itself acquires the structure of a crystal graph by requiring that the bijection is a crystal isomorphism. This trivial fact makes it obvious that the set of Gelfand-Tsetlin patterns has the structure of a crystal.
However to the best of our knowledge this crystal structure has not been made explicit. That is the goal of this paper.

A related result from \cite{WatYam2019} expresses the string length functions in terms of the entries of the Gelfand-Tsetlin patterns, and this data is enough to theoretically determine the other crystal data, see Remark \ref{rem:watyam}.

Our formulas (see Definition \ref{def:GT-crystal-data}) for the crystal data (raising and lowering operators, weight and string length functions) defined on Gelfand-Tsetlin patterns have several advantages over the description in terms of semistandard Young tableaux.
First of all, they are given by simple arithmetic expressions involving taking the maximum over a set of integers computed directly from the entries of the pattern at hand.
There is no need to apply a row- or column-reading function, and subsequently apply a cumbersome signature rule to the result, before the actual operation can be performed.

Secondly, particularly interesting is that the expressions are in fact given by tropical (Laurent) polynomials in the pattern entries. It is already well-known that crystal basis theory has deep connections to tropical mathematics, see e.g. \cite{bumpschilling}. It would be interesting to better understand the appearance of tropical polynomials in these expressions.

And lastly, it may be conjectured that similar tropical expressions can be written down in other types, in particular for representations of the symplectic and orthogonal Lie algebras, for which analogs of Gelfand-Tsetlin bases exist \cite{geltse1950b}\cite{zelobenko1973}\cite{molev}. Moreover, we expect that analogs of these formulas can be used to construct crystal bases for representations of some associative algebras analogous to enveloping algebras of Lie algebras, such as certain Galois orders \cite{FutOvs2010}.

This paper is organized as follows.
In Section \ref{sec:preliminaries} we fix some notation and terminology and recall some well-known results regarding semistandard Young tableaux and Gelfand-Tsetlin patterns.
In Section \ref{sec:crystal-structure} we state and prove the first theorem, in which we give explicit formulae for the crystal operators on $\GTP(n, \lambda)$ and prove that these equip the set with a crystal basis structure.  
In Section \ref{sec:crystal-isomorphism} we prove the second theorem, stating that the natural bijection between the sets $\SSYT(n,\lambda)$ and $\GTP(n, \lambda)$ is an isomorphism of crystals.  The proof relies on the fact that, by the nature of the bijection, we may obtain various useful combinatorial data about a given semistandard tableaux by examining certain sums and differences of the corresponding pattern entries, see Lemma \ref{lem:main2}. 
In Section \ref{sec:example} we give an example to illustrate how to apply the formulas.

\section*{Acknowledgements}
The first author gratefully acknowledges support from Simons Collaboration Grant for Mathematicians, award number 637600. This work was inspired by a question of G. Benkart.

\section{Preliminaries}
\label{sec:preliminaries}

In this section we recall some well-known definitions and results from the literature that we will use.

\subsection{Crystals}
We follow \cite{hongkang}. Let $X=\left(A,\, \Pi=\left\{\al_i\right\}_{i\in I},\, \Pi^\vee=\left\{\al_i^\vee\right\}_{i\in I},\, P,\, P^\vee\right)$ be a Cartan datum with finite index set $I$.
\begin{Remark}
The only Cartan datum we will use in this paper is $A_{n-1}$, where $I=\{1,2,\ldots,n-1\}$, the weight lattice $P$ is the abelian group generated by $\{\Be_i\}_{i=1}^n$ subject to $\Be_1+\Be_2+\cdots+\Be_n=0$, $\al_i=\Be_i-\Be_{i+1}$, and $\al_j^\vee\in\Hom_\Z(P,\Z)$ are given by $\langle \Be_i,\al_j^\vee\rangle = \delta_{ij}-\delta_{i,j+1}$.
\end{Remark}

\begin{Definition}
A \emph{crystal} of type $X$ is a non-empty set $\mathcal{B}$ together with maps 
\begin{align*}
\wt: \mathcal{B} &\rightarrow P, \\ 
\te_i, \tf_i: \mathcal{B} &\rightarrow \mathcal{B} \sqcup \{ 0 \}, \quad i\in I,\\
\varepsilon_i, \varphi_i: \mathcal{B} &\rightarrow \mathbb{Z} \sqcup \{ -\infty \},\quad i\in I, 
\end{align*}
satisfying for all $b,b'\in\mathcal{B}$ and $i\in I$:
\begin{enumerate}[{\rm (i)}]
\item $\tf_i(b)=b'$ if and only if $b=\te_i(b')$, in which case
\[ \wt(b') = \wt(b) - \alpha_i, \qquad
  \ep_i(b') = \ep_i(b) + 1, \qquad 
  \varphi_i(b') = \varphi_i(b) - 1\]
  and we write
        \[b\overset{i}{\longrightarrow} b'.\]
\item $\varphi_i(b) = \ep_i(b) + \langle \wt(b), \alpha_i^\vee \rangle$. In particular, $\varphi_i(b) = -\infty$ if and only if $\ep_i(b) = -\infty$.
\item If $\varphi_i(b) = \ep_i(b)=-\infty$, then $\te_i(b) = \tf_i(b) = 0$.
\end{enumerate}
The cardinality of $\mathcal{B}$ is the \emph{degree} of the crystal, $\op{wt}$ is called the \emph{weight map}, $\te_i$ and $\tf_i$ are called \emph{crystal operators}, and $\varphi_i$ and $\varepsilon_i$ are called \emph{string length functions}.
\end{Definition}

\begin{Definition}
Let $\mathcal{B}_1$ and $\mathcal{B}_2$ be crystals of type $X$.
A \emph{morphism} $\Psi: \mathcal{B}_1 \rightarrow \mathcal{B}_2$
is a map $\Psi:\mathcal{B}_1\sqcup\{0\}\to \mathcal{B}_2\sqcup\{0\}$ such that $\Psi(0)=0$ and for all $b,b'\in\Psi^{-1}(\mathcal{B}_2)$ and $i\in I$:
\begin{enumerate}[{\rm (i)}]
\item $\op{wt}(\Psi(b))=\op{wt}(b),\quad \ep_i(\Psi(b))=\ep_i(b),\quad \varphi_i(\Psi(b))=\varphi_i(b)$.
\item If $b\overset{i}{\longrightarrow} b'$ then $\Psi(b)\overset{i}{\longrightarrow}\Psi(b')$.
\end{enumerate}
If moreover $\Psi$ is bijective as a function $\mathcal{B}_1\sqcup\{0\}\to\mathcal{B}_2\sqcup\{0\}$, then $\Psi$ is an \emph{isomorphism}.
\end{Definition}

\subsection{Partitions, Young diagrams and semistandard Young tableaux}
\paragraph{1} For a fixed integer $N \geq 0$, a \emph{partition} of $N$ is a sequence $\lambda = (\lambda_1, \lambda_2, \dots)$ of integers $\lambda_i$ such that $\lambda_1 \geq \lambda_2 \geq \dots \geq 0$ and $|\lambda|:=\Sigma_{i\geq 1} \lambda_i = N$. The \emph{length} of a partition $\lambda$, $\ell(\lambda)$, is equal to the highest index $i$ for which $\lambda_i > 0$. For $1 \leq i \leq \ell(\lambda)$, the $\lambda_i$ are called the \emph{parts} of $\lambda$.  Let $\mathcal{P}(N)$ be the set of partitions of $N$ and put $\mathcal{P} = \bigcup_{N=0}^\infty \mathcal{P}(N)$. \\

\paragraph{2} If $\lambda$ is a partition of $N$, the \emph{Young diagram} $\YD(\lambda)$ is a left-justified collection of boxes where the $i$th row has $\lambda_i$ boxes.  The \emph{shape} of a Young diagram is its partition $\lambda$. A \emph{tableau} is a Young diagram whose boxes are filled with elements from an alphabet.
\[
\YD((3,2,2,1)) = \text{ }
    \ydiagram{3,2,2,1}
\] 

\paragraph{3} A \emph{subdiagram} is a Young diagram $\YD(\lambda')$ that is contained in Young diagram $\YD(\lambda)$.  A \emph{skew diagram} $\YD(\lambda/\lambda')$ is the diagram obtained by subtracting a subdiagram $\lambda'$ from $\lambda$.
\[
\YD((4,2,2,1)/(2,2)) = \text{ }
\begin{ytableau}
    \none & \none & & \\
        \none & \none \\
        & \\
        \\
    \end{ytableau}
\] 

\paragraph{4} A \emph{semistandard Young tableau} (SSYT) of shape $\lambda$ and rank $n$ is a tableau of shape $\lambda$ where the boxes are filled with entries from the alphabet is $[n]=\{1,2,\dots, n\}$ so that each row is weakly increasing from left to right and each column is strictly increasing from top to bottom.
\[
T = \text{ }
    \begin{ytableau}
1 & 1 & 3 \\
        2 & 3 \\
        3 & 4 \\
        4
\end{ytableau}
\]
Let $\SSYT(n,\lambda)$ denote the set of SSYTs of rank $n$ and shape $\lambda$.

\subsection{Crystal structure on $\SSYT(n,\lambda)$}
We recall the crystal structure on semistandard Young tableaux. For details, see e.g. \cite{hongkang},\cite{bumpschilling}.

Let $n$ be a positive integer and $\lambda$ be a partition of length at most $n$. Let $T\in\SSYT(n,\lambda)$. The \emph{far-eastern reading} of $T$, denoted $\mathrm{FarEast}(T)$ is the $|\la|$-tuple of letters read off from $T$, reading columns from right to left and each column top to bottom.
The map $\mathrm{FarEast}:\SSYT(n,\la)\to\{1,2,\ldots,n\}^{|\la|}$ is injective and we denote the inverse map by $\mathrm{FarEast}^{-1}$, defined on the image of $\mathrm{FarEast}$.

For $i\in\{1,2,\ldots,n-1\}$, the \emph{$i$-bracketing} of a tuple of letters $x=(x_1,x_2,\ldots,x_{|\la|})$, denoted in this paper by $[x]_i$ is obtained by crossing out the right-most $i$ having at least one $i+1$ to the right of it, in which case we also cross out the leftmost of those $i+1$'s, and repeating this recursively (ignoring crossed out entries) until $(i,i+1)$ is not a subsequence.
The crossed out $i$'s and $(i+1)$'s in $x$ are said to be ($i$-)\emph{bracketed}.
Any remaining $i$'s or $(i+1)$'s in $x$ are ($i$-)\emph{unbracketed}.

\begin{Definition}[Crystal structure on $\SSYT(n,\la)$, \cite{hongkang},\cite{bumpschilling}] Let $n$ be a positive integer and $\la$ a partition with $n$ or fewer parts. Let $P$ be the weight lattice of type $A_{n-1}$.
Define for $i\in\{1,2,\ldots,n-1\}$ and $T\in\SSYT(n,\la)$:
\begin{align*}
\wt(T)&=N_1(T)\boldsymbol{e}_1+N_2(T)\boldsymbol{e}_2+\cdots+N_n(T)\boldsymbol{e}_n,\text{ where $N_i(T)=\#$boxes in $T$ containing $i$,} \\
\varphi_i(T)&=\text{number of $i$-unbracketed  $i$'s in $[\mathrm{FarEast}(T)]_i$,}\\
\varepsilon_i(T)&=\text{number of $i$-unbracketed $(i+1)$'s in $[\mathrm{FarEast}(T)]_i$,}\\
\widetilde{f}_i(T) &=
\begin{cases}
\mathrm{FarEast}^{-1}\big(\text{change leftmost $i$ in $[\mathrm{FarEast}(T)]_i$ to $i+1$}\big), & 
\text{if $\varphi_i(T)>0$,}
\\
0,& \text{otherwise,}
\end{cases}\\
\widetilde{e}_i(T) &=
\begin{cases}
\mathrm{FarEast}^{-1}\big(\text{change rightmost  $i+1$ in $[\mathrm{FarEast}(T)]_i$ to $i$}\big), & 
\text{if $\ep_i(T)>0$,}\\
0,& \text{otherwise.}
\end{cases}
\end{align*}
\end{Definition}

\begin{Theorem}[See e.g. \cite{hongkang}]
Let $n$ be a positive integer and $\lambda$ be a partition with $n$ or fewer parts. The set $\SSYT(n,\la)$ equipped with the above maps $\wt,\varphi_i,\varepsilon_i,\widetilde{f}_i,\widetilde{e}_i$ constitutes a crystal of type $A_{n-1}$.
\end{Theorem}

\begin{Example}
Let $n=4$ and $\la=(5,2,2)$ and consider
\[
T =
\begin{ytableau}
1 & 2 & 2 & 2 & 3 \\
3 & 3 \\
4 & 4 
\end{ytableau}
\]
Let us compute $\varphi_2(T)$ and $\tilde{f}_2(T)$.
The far-eastern reading of $T$ is
\[
\mathrm{FarEast}(\La)=(3,2,2,2,3,4,1,3,4).
\]
To find the $2$-bracketing of this, first cross out the rightmost $2$ having a $3$ somewhere to its right, and also cross out the leftmost of those $3$'s: 
$(3,2,2,\xcancel{2},\xcancel{3},4,1,3,4)$. Repeating this step once more, ignoring crossed out entries, we obtain
$(3,2,\xcancel{2},\xcancel{2},\xcancel{3},4,1,\xcancel{3},4)$
at which point the $2$-bracketing is finished, as $(2,3)$ is not a subsequence anymore (ignoring crossed out entries).
So
\[
[(3,2,2,2,3,4,1,3,4)]_2=(3,2,\xcancel{2},\xcancel{2},\xcancel{3},4,1,\xcancel{3},4).
\]
Now we can compute:
\[
\varphi_2(T)
=\text{number of $2$-unbracketed $2$'s in 
$(3,2,\xcancel{2},\xcancel{2},\xcancel{3},4,1,\xcancel{3},4)$}=1
\]

\begin{align*}
\widetilde{f}_2(T) &= \mathrm{FarEast}^{-1}\big(\text{change rightmost $2$ in $(3,2,\xcancel{2},\xcancel{2},\xcancel{3},4,1,\xcancel{3},4)$
 to $3$}\big) \\
 &= \mathrm{FarEast}^{-1}\big( (3,3,\xcancel{2},\xcancel{2},\xcancel{3},4,1,\xcancel{3},4)\big) 
=\begin{ytableau}
1 & 2 & 2 & 3 & 3 \\
3 & 3 \\
4 & 4 
\end{ytableau}
\end{align*}
Note that when applying $\mathrm{FarEast}^{-1}$ we ignore the bracketing.
\end{Example}

\begin{Remark}\label{def:tbrac}
The $i$-bracketing can be described directly on the tableaux $T$ as follows. Go through all the columns of $T$ from left to right and do the following. If the column contains an $i$ and there is a thus-far-unbracketed $i+1$ in the same column, or in a column further to the left, then cross out that $i$ along with the rightmost of those $i+1$'s. 
Then $\varphi_i(T)$ is the number of $i$-unbracketed $i$'s in $T$;
$\ep_i(T)$ is the number of $i$-unbracketed $i+1$'s in $T$;
$\tf_i(T)$ is obtained from $T$ by changing the rightmost $i$-unbracketed $i$ in $T$ to $i+1$;
$\te_i(T)$ is obtained from $T$ by changing the leftmost $i$-unbracketed $i+1$ in $T$ to $i$.
\end{Remark}

\subsection{Gelfand-Tsetlin patterns}
Let $n$ be a positive integers and $\lambda$ be a partition with $n$ or fewer parts.
For us, a \emph{Gelfand-Tsetlin pattern} (GTP) with $n$ rows and top row $\lambda$ is a triangular array of integers
\[
\Lambda = 
    \begin{Bmatrix}
    \lambda^{(n)}_1 & & \lambda^{(n)}_2 & & \lambda^{(n)}_3 & & \cdots & & \lambda^{(n)}_n \\
    & \lambda^{(n-1)}_1 & & \lambda^{(n-1)}_2 & & \cdots & & \lambda^{(n-1)}_{n-1} \\
    & & \lambda^{(n-2)}_1 & & \cdots & & \lambda^{(n-2)}_{n-2} \\
    & & & \ddots & & \adots \\
    & & & & \lambda^{(1)}_{1}
\end{Bmatrix}
\]
where 
\begin{enumerate}[\rm(i)]
\item $\lambda^{(i)}_j \in \mathbb{Z}_{\geq 0}$ for $1\le j\le i\le n$,
\item $\lambda^{(i+1)}_j \geq \lambda^{(i)}_j \geq \lambda^{(i+1)}_{j+1}$ for $1\le j\le i\le n-1$,
\item $\lambda^{(n)}_i = \lambda_i$ for $1\le i\le n$.
\end{enumerate}
Condition (ii) is known as the \emph{interleaving condition}. By convention we set $\lambda^{(i)}_j=0$ if not $1\le j\le i\le n$.
  Let $\GTP(n,\lambda)$ denote the set of Gelfand-Tsetlin patterns with $n$ rows and top row $\lambda$.

\subsection{Bijection between tableaux and patterns}
\label{sec:natural-bijection}
There is a well-known and natural bijection $\mathcal{T}$ between $\GTP(n,\lambda)$ and $\SSYT(n,\lambda)$. 
Given a Gelfand-Tsetlin pattern $\Lambda \in \GTP(n,\lambda)$, we obtain a tableau $T=\mathcal{T}(\La) \in \SSYT(n,\lambda)$ by inserting $i$ into the squares of the skew diagram $\YD(\lambda^{(i)} / \lambda^{(i-1)})$, for $i=1,2,\dots,n$, where by convention $\lambda^{(0)}$ is the empty partition.
Conversely, given $T \in \SSYT(n,\lambda)$, we obtain a pattern $\Lambda=\mathcal{T}^{-1}(T) \in \GTP(n,\lambda)$ as follows. Define the top row $\lambda^{(n)}$ of $\Lambda$ to be the shape of $T$. That is, $\la^{(n)}=\lambda$. Then, delete all boxes from $T$ containing the symbol $n$ to obtain tableau $T^{(n-1)}$ and define the next row $\la^{(n-1)}$ of $\Lambda$ to be the shape of $T^{(n-1)}$. Continue in this fashion until all the boxes of $T$ have been deleted. Then all the rows of $\Lambda$ have been specified.

\section{Crystal structure on Gelfand-Tsetlin patterns}
\label{sec:crystal-structure}
In this section we prove the first main result of the paper.
We equip the set of Gelfand-Tsetlin patterns with explicit crystal data, and prove that this makes $\GTP(n,\la)$ into a crystal of type $A_{n-1}$. To define the crystal data we will need some notation.

\begin{figure}
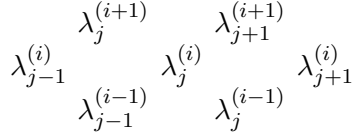

\[
\begin{matrix}
            & \la^{(i+1)}_j   &         & \la^{(i+1)}_{j+1} &  \\
\la^{(i)}_{j-1} &               & \la^{(i)}_j  &               & \la^{(i)}_{j+1}\\
            & \la^{(i-1)}_{j-1} &           & \la^{(i-1)}_j & 
\end{matrix}
\]
\caption{Part of a Gelfand-Tsetlin pattern}
\end{figure}

\begin{figure}
\centering
\parbox{2.5in}{
\begin{center}
$\begin{matrix}
     & &  & -\lambda^{(i+1)}_{j+1} &  &  & \\
     &  & {\color{red} \swarrow}  &  & {\color{red} \nwarrow} &  & \\
     & +\lambda^{(i)}_{j} &  &  &  & +\lambda^{(i)}_{j+1} & \\
     &  & {\color{red} \searrow} &  & {\color{red} \nearrow} &  & \\
     &  &  & -\lambda^{(i-1)}_{j} &  &  & 
\end{matrix}$\\[1em]
\textsc{Figure} \ref{fig:diamond}a. $a^{(i)}_j(\La)$\end{center}
}
\qquad
\parbox{2.5in}{
\begin{center}
$\begin{matrix}
      & &  & +\lambda^{(i+1)}_{j} &  &  &  \\
     &  & {\color{red} \nearrow}  &  & {\color{red} \searrow} &  & \\
     & -\lambda^{(i)}_{j-1} &  &  &  & -\lambda^{(i)}_{j} & \\
     &  & {\color{red} \nwarrow} &  & {\color{red} \swarrow} &  & \\
     &  &  & +\lambda^{(i-1)}_{j-1} &  &  & 
\end{matrix}$\\[1em]
\textsc{Figure} \ref{fig:diamond}b. $b^{(i)}_j(\La)$\end{center}
}
\caption{Computing diamond numbers in Gelfand-Tsetlin patterns}
\label{fig:diamond}
\end{figure}

We introduce the following \emph{diamond numbers},
which are alternating sums around a diamond shape in $\La$ starting at $\la^{(i)}_j$:
\begin{subequations}\label{eq:diamonds}
\begin{align}
a^{(i)}_j(\Lambda) &:= \lambda^{(i)}_j - \lambda^{(i-1)}_j +  \lambda^{(i)}_{j+1} - \lambda^{(i+1)}_{j+1}, \quad 0\le j\le i, \\
b^{(i)}_j(\Lambda) &:= - \lambda^{(i)}_j + \lambda^{(i-1)}_{j-1} - \lambda^{(i)}_{j-1} + \lambda^{(i+1)}_j, \quad 1\le j\le i+1,
\end{align}
\end{subequations}
where by convention $\la^{(i)}_j=0$ if not $1\le j\le i$.
Note that,
\begin{equation}
b_j^{(i)}(\Lambda)=-a_{j-1}^{(i)}(\La),\quad 1\le j\le i+1,
\end{equation}
and, by the interleaving conditions,
\begin{equation}\label{eq:ab-ineqs}
a_0^{(i)}(\La)\le 0,\quad b^{(i)}_{i+1}(\La)\le 0.
\end{equation}
For notational convenience we put
\begin{equation}
a_j^{(i)}(\Lambda)=0 \;\; \forall j>i, \qquad 
b_j^{(i)}(\Lambda)=0 \;\; \forall j>i+1.
\end{equation}
Next, define these \emph{diamond-sums}:
\begin{align}
\label{eq:diamond-sumA}
A^{(i)}_j(\La) &:= \sum_{k=j}^{i} a_k^{(i)}(\Lambda),\quad 0\le j\le i, \\
\label{eq:diamond-sumB}
B^{(i)}_j(\La) &:= \sum_{k=1}^j b_k^{(i)}(\Lambda), \quad 1\le j\le i+1.
\end{align}
Note that \eqref{eq:ab-ineqs} imply
\begin{equation}
A_0^{(i)}(\La)\le A_1^{(i)}(\La),\quad B_{i+1}^{(i)}(\La)\le B^{(i)}_i(\La).
\end{equation}
The following relation will be useful:
\begin{equation} \label{eq:AB-relation}
A_0^{(i)}(\La)=A_j^{(i)}(\La)-B_j^{(i)}(\La)=-B_{i+1}^{(i)}(\La)\qquad
\forall j\in\{0,1,\ldots,i+1\}.
\end{equation}

\begin{Remark}
\label{rem:watyam}
In \cite{WatYam2019} the authors give the formula for the $\mathbf{i}_A$-string datum,
where $\mathbf{i}_A$ is the reduced long word $(1,2,1,3,2,1,\ldots,n-1,n-2,\ldots,1)$. 
Converting their notation (their $a_{ij}$ is our $\la^{(n+i-j)}_i$) gives
the formula
$d_{i,j}(\Lambda)=\sum_{m=1}^{j-i}(\la^{(j)}_m-\la^{(j-1)}_m),\qquad 1\le i<j\le n$.
\end{Remark}

If $\La\in\GTP(n,\la)$, let $\La\pm\Delta^{(i)}_j$ denote the array of integers obtained from $\La$ by replacing $\la^{(i)}_j$ by $\la^{(i)}_j\pm 1$. (In general, the resulting array is not a valid Gelfand-Tsetlin pattern.)

\begin{Definition} \label{def:GT-crystal-data}
Let $P$ be the weight lattice of type $A_{n-1}$. Put $\omega_i=\sum_{j=1}^i\boldsymbol{e}_j$. Define for any $\Lambda\in\GTP(n, \lambda)$ and $i\in\{1,2,\ldots,n-1\}$:
\begin{align}
\wt(\La) &
=\sum_{j=1}^n \big(\sum_{k=1}^j \lambda^{(j)}_k - \sum_{k=1}^{j-1} \la^{(j-1)}_k\big)\boldsymbol{e}_j  \label{eq:wt} \\
&=A_0^{(1)}(\La)\omega_1+A_0^{(2)}(\La)\omega_2+\cdots+A_0^{(n)}(\La)\omega_n \nonumber  \\
&=-\big(B_{2}^{(1)}(\La)\omega_1+B_{3}^{(2)}(\La)\omega_2+\cdots+B_{n+1}^{(n)}(\La)\omega_n\big); \nonumber \\
\varphi_i(\Lambda) &= \max \big\{ A_1^{(i)}(\Lambda), A_2^{(i)}(\Lambda),\ldots, A_i^{(i)}(\Lambda) \big\};\\
\varepsilon_i(\Lambda) &= \max \big\{ B_1^{(i)}(\Lambda), B_2^{(i)}(\Lambda), \dots, B_i^{(i)}(\Lambda) \big\};\\
\widetilde{f}_i (\Lambda) &= \begin{cases}
 \Lambda - \Delta^{(i)}_\ell, &\text{if $\varphi_i(\La)>0$,} \\
 0, &\text{if $\varphi_i(\Lambda)=0$,}
 \end{cases}\\
& \quad \text{where }\ell = \max\big\{j\in\{1,2,\ldots,i\}\mid A^{(i)}_j(\Lambda)=\varphi_i(\Lambda) \big\};
 \nonumber \\
\widetilde{e}_i (\Lambda) &= 
\begin{cases}
 \Lambda + \Delta^{(i)}_\ell, &\text{if $\varepsilon_i(\La)>0$,}\\
 0, &\text{if $\varepsilon_i(\Lambda)=0$,}
\end{cases} \\
& \quad \text{where } \ell = \min\big\{j\in\{1,2,\ldots,i\}\mid B^{(i)}_j(\Lambda)=\varepsilon_i(\Lambda) \big\}. \nonumber
\end{align}
\end{Definition}

\begin{Theorem}
Let $n$ be any positive integer and $\la$ be a partition with $n$ or fewer parts. Then the set $\GTP(n,\la)$ of all Gelfand-Tsetlin patterns with $n$ rows and top row $\la$, equipped with the crystal data $\wt,\tf_i,\te_i,\varphi_i,\varepsilon_i$ as in Definition \ref{def:GT-crystal-data}, is a crystal of type $A_{n-1}$.
\end{Theorem}

\begin{proof}
Let $i\in\{1,2,\ldots,n-1\}$ be arbitrary.
First we show that if 
$\La\in\GTP(n,\la)$ is such that $\varphi_i(\La)>0$,
then $\tf_i(\Lambda)$ is a valid Gelfand-Tsetlin pattern. Let 
$\ell=\max\{j\in\{1,2,\ldots,i\}\mid A_j^{(i)}=\varphi_i(\La)\}$.
Then by definition, $\tf_i(\La)=\Lambda-\Delta^{(i)}_\ell$ which has integer entries, the top row still equals $\la$ (since $i<n$), and the interleaving conditions hold everywhere except possibly near the $\la_\ell^{(i)}$ entry. More precisely, we must show show the following inequalities hold:

\begin{equation} \label{eq:crystal-pf-inequalities}
\begin{matrix}
\la_\ell^{(i+1)} &     &  &     & \la_{\ell+1}^{(i+1)} \\
                 & \mathbin{\rotatebox[origin=c]{-45}{$\ge$}} &  & \mathbin{\rotatebox[origin=c]{45}{$\ge$}} & \\
                &     &\la_\ell^{(i)}-1 & & \\
                 & \mathbin{\rotatebox[origin=c]{45}{$\ge$}} &  & \mathbin{\rotatebox[origin=c]{-45}{$\ge$}} & \\
\la_{\ell-1}^{(i-1)} &     &  &     & \la_\ell^{(i-1)}
\end{matrix}
\end{equation}

We have that $\varphi_i(\La)=A_\ell^{(i)}(\La) = a_{\ell}^{(i)}(\La)+a_{\ell+1}^{(i)}(\La)+\cdots+a_i^{(i)}(\La)$. 
Note that $a_\ell^{(i)}(\La)>0$, otherwise $j=\ell+1$ would satisfy $A_j^{(i)}(\La)=\varphi_i(\La)$  (we can't have $A_{\ell+1}^{(i)}(\La)>\varphi_i(\La)$ by the definitions of $\varphi_i$ and $\ell$) contradicting maximality of $\ell$.
Now, $a_\ell^{(i)}>0$ is equivalent to 
\begin{equation}\label{eq:crystal-pf-eq0}
\la_\ell^{(i)}-\la_\ell^{(i-1)}+\la_{\ell+1}^{(i)}-\la_{\ell+1}^{(i+1)}>0
\end{equation}
by definition of $a_\ell^{(i)}$.
Since all entries of $\La$ are integers, \eqref{eq:crystal-pf-eq0} implies that
\begin{equation}\label{eq:crystal-pf-eq1}
\la_\ell^{(i)}-1 \ge \la_\ell^{(i-1)}+\la_{\ell+1}^{(i+1)}-\la_{\ell+1}^{(i)}.
\end{equation}
By the interleaving condition for $\La$,
\begin{equation}\label{eq:crystal-pf-eq2}
\la_{\ell+1}^{(i+1)}\ge \la_{\ell+1}^{(i)},
\quad\text{ and }\quad \la_{\ell}^{(i-1)}\ge \la_{\ell+1}^{(i)}.
\end{equation}
Combining \eqref{eq:crystal-pf-eq1} and \eqref{eq:crystal-pf-eq2} we obtain
\begin{equation}
\la_{\ell}^{(i)}-1\ge\la_\ell^{(i-1)}\quad\text{and}\quad \la_\ell^{(i)}-1\ge \la_{\ell+1}^{(i+1)}
\end{equation}
which are the two rightmost inequalities in \eqref{eq:crystal-pf-inequalities}.
The two leftmost inequalities in \eqref{eq:crystal-pf-inequalities} are trivial since 
$\la_\ell^{(i+1)}\ge \la^{(i)}$ and $\la_{\ell-1}^{(i-1)}\ge \la_\ell^{(i)}$ by the interleaving conditions for $\La$.
This shows that if $\varphi_i(\La)>0$ then $\tf_i(\La)\in\GTP(n,\la)$.

Next, suppose that $\ep_i(\La)>0$. We must show that $\te_i(\La)\in\GTP(n,\la)$. We have $\ep_i(\La)=\max\{B_1^{(i)}(\La),\ldots,B_i^{(i)}(\La)\}$.
Let $\ell=\min\{j\in\{1,2,\ldots,i\}\mid B_j^{(i)}(\La)=\ep_i(\La)\}$.
Then $\ep_i(\La)=B_{\ell}^{(i)}=b_1^{(i)}(\La)+b_2^{(i)}(\La)+\cdots+b_\ell^{(i)}(\La)$. As before, $b_\ell^{(i)}(\La)>0$ by the minimality of $\ell$. So
\begin{equation}\label{eq:crystal-pf-eq5}
-\la_{\ell-1}^{(i)}+\la_{\ell-1}^{(i-1)}-\la_{\ell}^{(i)}+\la_{\ell}^{(i+1)}>0.
\end{equation}
We have $\te_i(\La)=\La+\Delta^{(i)}_\ell$ and hence we must show that
\begin{equation} \label{eq:crystal-pf-inequalities-e}
\begin{matrix}
\la_\ell^{(i+1)} &     &  &     & \la_{\ell+1}^{(i+1)} \\
                 & \mathbin{\rotatebox[origin=c]{-45}{$\ge$}} &  & \mathbin{\rotatebox[origin=c]{45}{$\ge$}} & \\
                &     &\la_\ell^{(i)}+1 & & \\
                 & \mathbin{\rotatebox[origin=c]{45}{$\ge$}} &  & \mathbin{\rotatebox[origin=c]{-45}{$\ge$}} & \\
\la_{\ell-1}^{(i-1)} &     &  &     & \la_\ell^{(i-1)}
\end{matrix}
\end{equation}
Analogously to the previous case, the rightmost two inequalities
$\la_\ell^{(i)}+1\ge \la_{\ell+1}^{(i+1)}$ and 
$\la_\ell^{(i)}+1\ge \la_{\ell+1}^{(i-1)}$ hold trivially by the interleaving conditions for $\La$.
By \eqref{eq:crystal-pf-eq5} we have
\begin{equation}
\la_\ell^{(i)}+1\le -\la_{\ell-1}^{(i)}+\la_{\ell-1}^{(i-1)}+\la_{\ell}^{(i+1)}
\end{equation}
which together with $\la_{\ell-1}^{(i)}\ge\la_{\ell-1}^{(i-1)}$ and $\la_{\ell-1}^{(i)}\ge\la_{\ell}^{(i+1)}$ which hold by the interleaving condition for $\La$, yields
the leftmost two inequalities in \eqref{eq:crystal-pf-inequalities-e}.
This shows that if $\ep_i(\La)>0$ then $\te_i(\La)\in\GTP(n,\la)$.

Next we show that property (i) in the definition of crystal holds.
First we show that $\tf_i(\Lambda)=\Lambda'$ iff $\te_i(\Lambda')=\Lambda$.
Suppose $\tf_i(\Lambda)=\Lambda'$. In particular $\varphi_i(\La)>0$.
Then we need to prove $\te_i(\Lambda')=\Lambda$.
We have $\Lambda'=\Lambda-\Delta^{(i)}_\ell$ where $\ell$ is defined by
\begin{equation}\label{eq:GT-crystal-pf-fe-ell}
\ell=\max\{j\in\{1,2,\ldots,i\}\mid A^{(i)}_j(\La)=\varphi_i(\La)\}.
\end{equation}
First we show that $\ep_i(\Lambda')>0$.
By definition, $\ep_i(\Lambda')=\max\{B_1^{(i)}(\La'),B_2^{(i)}(\La'),\ldots,B_i^{(i)}(\La')\}$. So it suffices to show that $B_j^{(i)}(\La')>0$ for some $j$. For $j=\ell$ we have:
\begin{equation}\label{eq:GT-crystal-pf-fe1}
B_\ell^{(i)}(\La')=b_1^{(i)}(\La')+b_2^{(i)}(\La')+\cdots+b_\ell^{(i)}(\La')=B^{(i)}_\ell(\La)+1
\end{equation}
since $b_j^{(i)}(\La')=b_j^{(i)}(\La)$ for $j=1,\ldots,i-1$ while $b_{\ell}^{(i)}(\La')=b_{\ell}^{(i)}(\La)+1$ by definition of $b_j^{(i)}(\La)$.
By \eqref{eq:AB-relation}, 
\begin{equation}\label{eq:GT-crystal-pf-fe2}
B_\ell^{(i)}(\La)+1 = A_\ell^{(i)}(\La)-A_0^{(i)}(\La)+1 = \varphi_i(\La)-A_0^{(i)}(\La)+1
\end{equation}
By definition of $\varphi_i$ we have 
\begin{equation}\label{eq:GT-crystal-pf-fe3}
\varphi_i(\La)-A_0^{(i)}(\La)\ge 0.
\end{equation}
Now \eqref{eq:GT-crystal-pf-fe1}-\eqref{eq:GT-crystal-pf-fe3} imply $B_\ell^{(i)}(\La')>0$, hence $\ep_i(\La')>0$.
It remains to be shown that $\te_i(\La')=\La$.
Since $\La=\La'+\Delta^{(i)}_\ell$, we have to show that 
\begin{equation}\label{eq:GT-crystal-pf-fe-mustshow}
\ell=\min\{j\in\{1,2,\ldots,i\}\mid B_j^{(i)}(\La')=\ep_i(\La')\}.
\end{equation}
For $1\le j<\ell$ we saw that $B_j^{(i)}(\La')=B_j^{(i)}(\La)$ and \eqref{eq:AB-relation} implies that $B_j^{(i)}(\La)<B_\ell^{(i)}(\La)$, while $B_\ell^{(i)}(\La')=1+B_\ell^{(i)}(\La)$.
So $\ep_i(\La')\ge B_\ell^{(i)}(\La')$ and we will show equality.
For $\ell<j\le i$ we have, by definition of $b_j^{(i)}(\La)$,
\begin{equation}\label{eq:GT-crystal-pf-fe4}
B_j^{(i)}(\La')=2+B_j^{(i)}(\La),
\end{equation}
and by \eqref{eq:AB-relation}, 
\begin{equation}\label{eq:GT-crystal-pf-fe5}
B_j^{(i)}(\La)=A_j^{(i)}(\La)-A_0^{(i)}(\La),
\end{equation}
while by definition of $\ell$, \eqref{eq:GT-crystal-pf-fe-ell}, we have
\begin{equation}\label{eq:GT-crystal-pf-fe6}
A_j^{(i)}(\La)-A_0^{(i)}(\La) < A_\ell^{(i)}(\La)-A_0^{(i)}(\La).
\end{equation}
Thus \eqref{eq:GT-crystal-pf-fe4}-\eqref{eq:GT-crystal-pf-fe6} imply that 
\begin{equation}\label{eq:GT-crystal-pf-fe7}
B_j^{(i)}(\La')\le 1+B_\ell^{(i)}(\La)=B_\ell^{(i)}(\La').
\end{equation}
Therefore $\ep_i(\La')=B_{\ell}^{(i)}(\La')$ and \eqref{eq:GT-crystal-pf-fe-mustshow} holds.

The converse is analogous but we provide some details for the sake  completeness. Suppose that $\te_i(\Lambda')=\Lambda$.
We need to show that $\tf_i(\Lambda)=\Lambda'$.
We have $\ep_i(\La')>0$ and $\La=\La'+\Delta_{\ell}^{(i)}$ where $\ell=\min\{j\in\{1,2,\ldots,i\}\mid B_j^{(i)}(\La')=\ep_i(\La')\}$. 
First we show $\varphi_i(\La)>0$ by showing $A_\ell^{(i)}(\La)>0$. We have $A_\ell^{(i)}(\La)=A_\ell^{(i)}(\La')+1$ and $A_\ell^{(i)}(\La')=\ep_i(\La')-B_{i+1}^{(i)}(\La')\ge 0$ by \eqref{eq:AB-relation}.
It remains to show $\tf_i(\La)=\La'$. Since $\La'=\La-\Delta_\ell^{(i)}$, this is equivalent to showing that $\ell=\max\{j\in\{1,2,\ldots,i\}\mid A_j^{(i)}(\La)=\varphi_i(\La)\}$.
For $\ell<j\le i$ we have $A_j^{(i)}(\La)=A_j^{(i)}(\La')=B_j^{(i)}(\La')-B_{i+1}^{(i)}(\La')\le B_{\ell}^{(i)}(\La')-B_{i+1}^{(i)}(\La')=A_\ell^{(i)}(\La')=A_\ell^{(i)}(\La)-1<A_\ell^{(i)}(\La)$.
So $\ell\le\max\{j\in\{1,2,\ldots,i\}\mid A_j^{(i)}(\La)=\varphi_i(\La)\}$. For $1\le j<\ell$ we have $A_j^{(i)}=2+A_j^{(i)}(\La')=2+B_j^{(i)}(\La')-B_{i+1}^{(i)}(\La')\le 1+B_\ell^{(i)}-B_{i+1}^{(i)}(\La')=1+A_\ell^{(i)}(\La')=A_\ell^{(i)}(\La)$.
This proves the desired equality.

Suppose now that $\tf_i(\La)=\La'$ and $\te_i(\La')=\La$ hold. In this case, all the entries of $\Lambda'$ equal those of $\Lambda$, except for one entry $\la'^{(i)}_\ell$ in the $i$th row which equals $\la^{(i)}_\ell-1$. Therefore 
\begin{align*}
\wt(\Lambda') &=\sum_{j=1}^n \big(\sum_{k=1}^j \lambda'^{(j)}_k - \sum_{k=1}^{j-1} \la'^{(j-1)}_k\big)\boldsymbol{e}_j \\
&=-\boldsymbol{e}_i+\boldsymbol{e}_{i+1}+
\sum_{j=1}^n \big(\sum_{k=1}^j \lambda^{(j)}_k
  - \sum_{k=1}^{j-1} \la^{(j-1)}_k\big)\boldsymbol{e}_j \\
&=\wt(\La)-\al_i
\end{align*}
which is equivalent to $\wt(\La')=\wt(\La)+\al_i$.

To conclude the proof of (i) we need to show 
$\varepsilon_i(\Lambda')=\varepsilon_i(\Lambda)+1$ and
$\varphi_i(\Lambda')=\varphi_i(\Lambda)-1$.
For $1\le j,\ell\le i$ and any $\La\in\GTP(n,\la)$ we have
\[
A^{(i)}_j(\La-\Delta_\ell^{(i)}) =
\begin{cases}
A_j^{(i)}(\La)  , & 0\le j<\ell, \\
A_j^{(i)}(\La)-1, & j=\ell, \\
A_j^{(i)}(\La)-2, & \ell<j\le i.
\end{cases}
\]
Suppose $\varphi_i(\La)>0$ and let $\ell=\max\{j\in\{1,2,\ldots,i\}\mid A_j^{(i)}=\varphi_i(\La)\}$.
Then for all $1\le j\le i$.
$A_j^{(i)}(\La-\Delta_\ell^{(i)})\le \varphi_i(\La)-1$ with equality for $j=\ell$.
Therefore $\varphi_i(\tf_i(\La))=\varphi_i(\La)-1$.

Let $1\le j,\ell\le i$. Then for any $\La\in\GTP(n,\la)$ we have
\[
B^{(i)}_j(\La+\Delta_\ell^{(i)})=
\begin{cases}
B_j^{(i)}(\La)   & 1\le j < \ell,\\
B_j^{(i)}(\La)-1 & j=\ell,\\
B_j^{(i)}(\La)-2 & \ell < j \le i.
\end{cases}
\]
Suppose $\ep_i(\La)>0$ and let $\ell=\min\{j\in\{1,2,\ldots,i\}\mid B_j^{(i)}(\La)=\ep_i(\La) \}$.
Then $B_j^{(i)}(\La+\Delta_\ell^{(i)})\le \ep_i(\La)-1$ with equality for $j=\ell$. Thus $\ep_i(\te_i(\La))=\ep_i(\La)-1$.

For property (ii), we verify that for all $\Lambda\in\GTP(n,\lambda)$ we have
\[\varphi_i(\Lambda)-\varepsilon_i(\Lambda) = \langle \wt(\Lambda), \al_i^\vee \rangle.\]
We have $\varphi_i(\Lambda)=\max\{A_1^{(i)},A_2^{(i)},\ldots,A_i^{(i)}\}$
and $\varepsilon_i(\Lambda)=\max\{B_1^{(i)},B_2^{(i)},\ldots,B_i^{(i)}\}$.
We will use relation \eqref{eq:AB-relation}.
Writing $A^{(i)}_k=A_k^{(i)}(\Lambda)$ for brevity we have
\begin{align*}
\varphi_i(\Lambda)-\varepsilon_i(\Lambda) &= 
\max\{A^{(i)}_1,A^{(i)}_2,\ldots,A^{(i)}_i\}-
 \max\{ A^{(i)}_1-A^{(i)}_0, A^{(i)}_2-A^{(i)}_0,\ldots, A^{(i)}_i-A^{(i)}_0\} \\ 
 &=\max\{A^{(i)}_1,A^{(i)}_2,\ldots,A^{(i)}_i\}-
 \max\{ A^{(i)}_1, A^{(i)}_2,\ldots, A^{(i)}_i\} + A^{(i)}_0 \\
 &=A^{(i)}_0 \\
 &=\sum_{k=0}^i (\la^{(i)}_k+\la^{(i)}_{k+1}-\la^{(i-1)}_k-\la^{(i+1)}_{k+1}) \\
 &=2\sum_{k=1}^i \la^{(i)}_k-\sum_{k=1}^{i-1}\la^{(i-1)}_k - \sum_{k=1}^{i+1}\la^{(i+1)}_k.
\end{align*}
On the other hand, using the first expression for the weight function, we have
\[
  \wt(\Lambda)=\sum_{j=1}^n \big(\sum_{k=1}^j \lambda^{(j)}_k - \sum_{k=1}^{j-1} \la^{(j-1)}_k\big)\boldsymbol{e}_j
\]
so using 
\[
  \langle \boldsymbol{e}_j, \al_i^\vee \rangle 
  = \langle \omega_j-\omega_{j-1}, \al_i^\vee \rangle 
  = \delta_{ji}-\delta_{j-1,i}
\]
 we get
\begin{align*}
 \langle \wt(\Lambda), \al_i^\vee \rangle &=
\langle \sum_{j=1}^n \big(\sum_{k=1}^j \lambda^{(j)}_k 
- \sum_{k=1}^{j-1} \la^{(j-1)}_k\big)\boldsymbol{e}_j ,  \al_i^\vee \rangle \\
&=
\sum_{j=1}^n \big(\sum_{k=1}^j \la^{(j)}_k -\sum_{k=1}^{j-1} \la^{(j-1)}_k \big)(\delta_{ji}-\delta_{j-1,i}) \\
&=2\sum_{k=1}^i \la^{(i)}_k - \sum_{k=1}^{i-1} \la^{(i-1)}_k -\sum_{k=1}^{i+1} \la^{(i+1)}_k
\end{align*}
This shows that $\varphi_i(\Lambda)-\varepsilon_i(\Lambda) = \langle \wt(\Lambda), \al_i^\vee \rangle$, which also equals the coefficient of $\omega_i$ in $\wt(\La)$, proving the second and third equality in \eqref{eq:wt}.

Lastly, since $\varphi_i(\Lambda)$ and $\epsilon_i(\Lambda)$ are never $-\infty$,
condition (iii) in the definition of a crystal is void.
\end{proof}

\section{Crystal isomorphism}
\label{sec:crystal-isomorphism}
In this section we prove our second main result which says that the natural bijection $\mathcal{T}$ from $\SSYT(n,\lambda)$ to $\GTP(n,\lambda)$ described in Section \ref{sec:natural-bijection} is an isomorphism of crystals.

We will let $T_i$ denote the $i$th row of a semistandard Young tableaux $T$, and $T_{\ge \ell}$ the subtableau obtained by deleting the first $\ell-1$ rows,and similarly for $T_{\le \ell}$:
\[
T = \begin{matrix} T_1 \\ T_2 \\ \vdots \\ T_n \end{matrix}
\qquad \qquad
T_{\ge \ell} = \begin{matrix} T_\ell \\ T_{\ell+1} \\ \vdots \\ T_n \end{matrix}
\qquad \qquad
T_{\le \ell} = \begin{matrix} T_1 \\ T_{2} \\ \vdots \\ T_\ell \end{matrix}
\]

The following counting lemma will be useful.

\begin{Lemma}\label{lem:main2}
Let $\La\in \GTP(n,\lambda)$ and $T=\mathcal{T}(\La)$.
\begin{enumerate}[{\rm (a)}]
\item \label{it:lem-main2-N}
For all integers $k$ with $1\le k\le n$, the number of letters $i$ in $T_k$ is equal to $\lambda^{(i)}_k - \lambda^{(i-1)}_k$.
\item \label{it:lem-main2-a}
 $a^{(i)}_j(\Lambda)$ counts the number of $i$'s in $T_j$ minus the number of $(i+1)$'s in $T_{j+1}$.  
\item \label{it:lem-main2-b}
 $b^{(i)}_j(\Lambda)$ counts the number of $(i+1)$'s in $T_j$ minus the number of $i$'s in $T_{j-1}$.  
\item \label{it:lem-main2-A}
 $A^{(i)}_\ell(\Lambda)$ counts the number of $i$'s in $T_{\ge \ell}$ minus the number of $(i+1)$'s in $T_{\ge \ell+1}$.
 \item \label{it:lem-main2-B}
 $B_{\ell}^{(i)}(\Lambda)$ counts the number of $(i+1)$'s in $T_{\le \ell}$ minus the number of $i$'s in $T_{\le \ell-1}$.
\end{enumerate}
\end{Lemma}

\begin{proof}
(a) The number of boxes in $T_k$ containing a letter from $\{ 1, 2, \dots, i \}$ is $\lambda^{(i)}_k$.
Then (b) and (c) are immediate by part (a) and the definitions,  \eqref{eq:diamonds}, of the diamond numbers.
Now (d) and (e) follow from parts (b) and (c).
\end{proof}

\begin{Theorem} \label{thm:main2}
Let $n$ be a positive integer and $\la$ a partition with $n$ or fewer parts.
The bijection $\mathcal{T}$ from $\GTP(n,\lambda)$ to $\SSYT(n,\lambda)$ given in Section \ref{sec:natural-bijection} is an isomorphism of crystals.
\end{Theorem}

\begin{proof}
Let $\La\in\GTP(n,\lambda)$, and let $T=\mathcal{T}(\La)$.

\underline{$\wt(\La)=\wt(T)$}:
For each $i\in\{1,2,\ldots,n-1\}$, by Lemma \ref{lem:main2}\eqref{it:lem-main2-N}, $\sum_{j=1}^i \la^{(i)}_j-\sum_{j=1}^{(i-1)} \la^{(i-1)}_j$ equals $N_i(T)$, since the letter $i$ cannot occur below the $i$th row in an SSYT.

Let $i\in \{1,2,\ldots,n-1\}$ be arbitrary. \emph{In the rest of the proof, ``bracketing'' refers to $i$-bracketing.} Put $A^{(i)}_k=A_k^{(i)}(\Lambda)$ and $B^{(i)}_k=B^{(i)}_k(\La)$ for brevity.

\underline{$\varphi_i(\Lambda)=\varphi_i(T)$:}
By definition, $\varphi_i(T)$ is the number of unbracketed $i$'s in 
$T$.
So $\varphi_i(T)\ge \varphi_i(T_{\ge j})$ for any $j\in\{1,2\ldots,i\}$.
Let $j_1\ge j_2\ge\cdots\ge j_k$ be all the rows of $T$ containing at least one unbracketed $i$.
Then $\varphi_i(T_\ge {j_1}) = A_{j_1}^{(i)}$ by Lemma \ref{lem:main2}\eqref{it:lem-main2-A}. Furthermore,
$A_{j_1}^{(i)}>A_{k}^{(i)}$ for $k=i, i-1, \ldots, j_1+1$.
Next, 
$\varphi_i(T_{\ge j_2}) = \varphi_i(T_{\ge{j_1}})+\varphi_i(T_{\ge {j_2}}/T_{\ge {j_1}}) = A_{j_1}^{(i)}+(A_{j_2}^{(i)}-A_{j_1}^{(i)})=A_{j_2}^{(i)}$. (Here $T_{\ge {j_2}}/T_{\ge {j_1}}$ denotes the subtableau of $T$ consisting of row $j_2$ through row $j_1-1$.)
And $A_{j_2}^{(i)}> A_k^{(i)}$ for $k=j_1, j_1-1,\ldots, j_2+1$.
Continuing recursively, we eventually obtain that
$\varphi_i(T)=\varphi_i(T_{\ge {j_k}}) = A_{j_k}^{(i)}>A_{j}^{(i)}$ for $j>j_k$.
It remains to be shown that $A^{(i)}_j\le A^{(i)}_{j_k}$ for $i=j_k-1,j_k-2,\ldots, 1$.
Since $j_k$ is the top row having unbracketed $i$'s, we have $A_j^{(i)}(\La_{\le j_k-1}) \le 0$ for $j=j_k-1, j_k-2,\ldots, 1$, where $\La_{\le r}$ is defined to  be $\mathcal{T}^{-1}(T_{\le r})$ for all $r$.
Since $A_j^{(i)}(\La) - A_{j_k}^{(i)}(\La) = A_j^{(i)}(\La_{\le j_k-1})$, this shows the required inequality.

\underline{$\varepsilon_i(\Lambda)=\varepsilon_i(T)$}:
This part can be proved completely analogously to the case of $\varphi_i$. But it also follows from the case of $\varphi_i$ and the fact that we already know that $\GTP(n,\la)$ and $\SSYT(n,\la)$ are crystals, and hence by property (ii) in the definition of crystal and that $\wt(\La)=\wt(T)$,
\[
\ep_i(\La) = \varphi_i(\La)-\langle\wt(\La),\al_i^\vee\rangle = \varphi_i(T)-\langle\wt(T),\al_i^\vee\rangle = \ep_i(T).
\]

\underline{$\mathcal{T}\big(\tf_i(\Lambda)\big)=\tf_i\big(\mathcal{T}(\Lambda)\big)$}:
We have seen already that $\varphi_i(\La)=\varphi_i(T)$.
Thus $\tf_i(\La)\neq 0$ iff $\tf_i(T)\neq 0$.
Suppose $\tf_i(\La)\neq 0$.
Put $\La'=\tf_i(\La)=\La-\Delta^{(i)}_\ell$, where 
 $\ell=\max\{i\in\{1,2,\ldots,i\}\mid A_j^{(i)}(\La)=\varphi_i(\La)\}$. By definition of the bijection $\mathcal{T}$, the SSYT $\mathcal{T}(\La')$ is obtained from $T$ by changing the rightmost $i$ in row $\ell$ to $i+1$. On the other hand, $\tf_i(T)$ is obtained by changing the rightmost unbracketed $i$ in $T$ to $i+1$. 
 So we must show that $\ell$ equals the row index of the rightmost unbracketed $i$ in $T$.
 First we show that there is an unbracketed $i$ in row $\ell$ of $T$. 
To do this we derive a series of equivalences. Let $j\in\{1,2,\ldots,i\}$ be arbitrary. Then:
\begin{align*}
&\text{$T$ has an unbracketed $i$ in row $j$}\\
\Leftrightarrow \; & \varphi_i(T_{\ge j})>\varphi_i(T_{\ge j+1}) \\
\Leftrightarrow \; & \varphi_i(\La_{\ge j})>\varphi_i(\La_{\ge j+1}) \quad\text{where $\La_{\ge k}:=\mathcal{T}^{-1}(T_{\ge k})$} \\
\Leftrightarrow \; & \max\{A^{(i)}_k(\La_{\ge j})\mid k=1,2,\ldots,i\} > \max\{A^{(i)}_k(\La_{\ge j+1})\mid k=1,2,\ldots,i\} \\
\Leftrightarrow \; & \max\{A^{(i)}_k(\La)\mid k=j,j+1,\ldots,i\} > \max\{A^{(i)}_k(\La)\mid k=j+1,j+2,\ldots,i\} \\
\Leftrightarrow \; & A^{(i)}_j(\La)> A^{(i)}_k(\La) \quad \text{ for all $k\in\{j+1,\, j+2,\,\ldots,\, i\}$.}
\end{align*}
The penultimate equivalence holds by the counting lemma, Lemma \ref{lem:main2}\eqref{it:lem-main2-A}, and that the first row of $T_{\ge j}$ is the $j$th row of $T$ and so on.
Now, by definition of $\ell$ we do indeed have
\[A^{(i)}_\ell(\La)>A^{(i)}_k(\La)\quad\text{for all $k\in\{\ell+1,\, \ell+2,\, \ldots,\, i\}$}\]
and therefore by the above series of equivalences there is at least one unbracketed $i$ in row $\ell$ of $T$.

It remains to show that $\ell$ is the row of the rightmost unbracketed $i$ in $T$. Since any $i$ directly to the right of an unbracketed $i$ is itself unbracketed, any unbracketed $i$ further to the right would have to occur among the top $\ell-1$ rows of $T$.
Any unbracketed $i$ among the top $\ell-1$ rows of $T$ would remain unbracketed when considered as an entry of the truncated tableau $T_{\le\ell-1}$.
So it suffices to show that $T_{\le \ell-1}$ has no unbracketed $i$'s, or equivalently, that $\varphi_i(T_{\le\ell-1})=0$. Let
$\La_{\le\ell-1}=\mathcal{T}^{-1}(T_{\le\ell-1})$. As previously shown,
$\varphi_i(T_{\le\ell-1})=\varphi_i(\La_{\le\ell-1})$. By Lemma \ref{lem:main2}\eqref{it:lem-main2-A},
for all $1\le j\le i$:
\[
A^{(i)}_j(\La_{\le\ell-1}) = A^{(i)}_j(\La)-A^{(i)}_\ell(\La)
\]
which is less than or equal to zero by definition of $\ell$. Hence $\varphi_i(\La_{\le\ell-1})=0$.

\underline{ $\mathcal{T}\big(\te_i(\Lambda)\big)=\te_i\big(\mathcal{T}(\Lambda)\big)$:}
We know that $\ep_i(\La)=\ep_i(T)$. Thus $\te_i(\La)=0$ iff $\te_i(T)=0$. Suppose that $\te_i(\La)\neq 0$. Put $\La'=\te_i(\La)=\La+\Delta^{(i)}_\ell$, where
\[\ell=\min\big\{j\in\{1,2,\ldots,i\}\mid B^{(i)}_j(\La)=\ep_i(\La)\big\}.\]
Also recall that
\[\ep_i(\La)=\max\{B^{(i)}_1(\La),\, B^{(i)}_2(\La),\, \ldots,\, B^{(i)}_i(\La)\}.\]
By definiton of the bijection $\mathcal{T}$, the SSYT $\mathcal{T}(\La')$ is obtained from $T$ by changing the leftmost $i+1$ in row $\ell$ of $T$ to $i$. On the other hand, $\te_i(T)$ is the SSYT obtained from $T$ by changing the leftmost unbracketed $i+1$ to $i$. So we must show that $\ell$ equals the row index of the row in $T$ which contains the leftmost unbracketed $i+1$.

First we show that row $\ell$ of $T$ contains an unbracketed $i+1$.
For this, we derive an equivalent condition.
For all $j\in\{1,2,\ldots,i\}$ we have:
\begin{align*}
&\text{$T$ contains an unbracketed $i+1$ in row $j+1$}\\
\Leftrightarrow \; & \ep_i(T_{\le j+1}) > \ep_i(T_{\le j}) \\
\Leftrightarrow \; & \ep_i(\La_{\le j+1}) > \ep_i(\La_{\le j}) \quad\text{where $\La_{\le k}:=\mathcal{T}^{-1}(T_{\le k})$}\\
\Leftrightarrow \; & \max\{B^{(i)}_k(\La_{\le j+1})\mid k=1,2,\ldots,i\}>\max\{B^{(i)}_k(\La_{\le j})\mid k=1,2,\ldots,i\}\\
\Leftrightarrow \; & B^{(i)}_{j+1}>B^{(i)}_k\quad\text{for all $k\in\{1,2,\ldots,j\}$}
\end{align*}
This condition holds for $j+1=\ell$ by definition of $\ell$. Thus $T$ contains an unbracketed $i+1$ in row $\ell$.

Next we show that no row of $T$ contains an unbracketed $i+1$ further to the left. Such a row $j$ would have to be below $\ell$, i.e. $j\ge i+1$.
By the above equivalences we would get
\[B^{(i)}_j(\La)> B^{(i)}_k(\La)\quad\text{for all $k\in\{1,2,\ldots,j-1\}$}.\]
In particular, $B^{(i)}_j(\La)>B^{(i)}_\ell(\La)$ which contradicts the definition of $\ell$. This finishes the proof that $\mathcal{T}(\te_i(\La))=\te_i(\mathcal{T}(\La))$.

\underline{Alternative proof that $\mathcal{T}\big(\te_i(\Lambda)\big)=\te_i\big(\mathcal{T}(\Lambda)\big)$:}
As is well-known, if a function between crystals preserve the string length functions and intertwines the $\tf_i$ crystal operators, the it automatically intertwines the $\te_i$ crystal operators. We illustrate this for the convenience of the reader.
We know that $\ep_i(\La)=\ep_i(T)$. Thus $\te_i(\La)=0$ iff $\te_i(T)=0$. Suppose that $\te_i(\La)\neq 0$.
Since $\varphi_i(\mathcal{T}(\te_i(\La))=\varphi_i(\te_i(\La))\ge 1$.
Thus we have
\begin{align*}
\mathcal{T}\big(\te_i(\Lambda)\big) &=
\te_i\tf_i\big(\mathcal{T}(\te_i(\Lambda))\big)\\
&=\te_i\mathcal{T}\big(\tf_i\te_i(\La)\big) \quad\text{by $\mathcal{T}\tf_i=\tf_i\mathcal{T}$}\\
&=\te_i\big(\mathcal{T}(\Lambda)\big).
\end{align*}
\end{proof}

\section{Example}
\label{sec:example}

\usetikzlibrary{arrows,quotes}

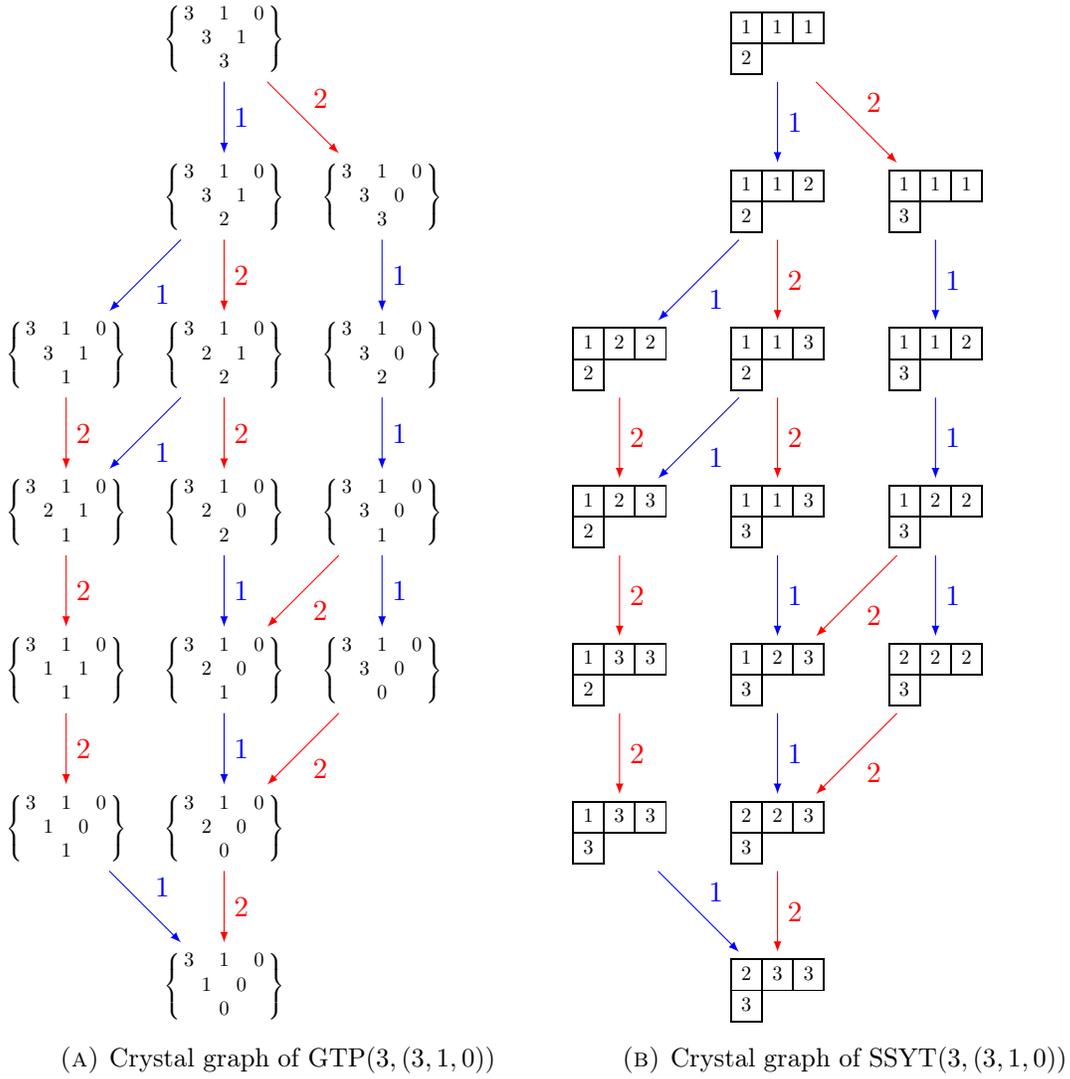
\begin{figure}
\centering
\begin{subfigure}[b]{0.49\textwidth}
\begin{tikzpicture}[>=latex,line join=bevel,scale=.7]
   \node (1) at (0,0) [draw,draw=none,scale=.7] {$\left\{\begin{array}{*{5}c} 3 & & 1 & & 0 \\ & 3 & & 1 & \\ & & 3 & & \end{array}\right\}$};
   \node (2) at (0,-3) [draw,draw=none,scale=.7] {$\left\{\begin{array}{*{5}c} 3 & & 1 & & 0 \\ & 3 & & 1 & \\ & & 2 & & \end{array}\right\}$};
   \node (5) at (0,-6) [draw,draw=none,scale=.7] {$\left\{\begin{array}{*{5}c} 3 & & 1 & & 0 \\ & 2 & & 1 & \\ & & 2 & & \end{array}\right\}$};
   \node (8) at (0,-9) [draw,draw=none,scale=.7] {$\left\{\begin{array}{*{5}c} 3 & & 1 & & 0 \\ & 2 & & 0 & \\ & & 2 & & \end{array}\right\}$};
   \node (11) at (0,-12) [draw,draw=none,scale=.7] {$\left\{\begin{array}{*{5}c} 3 & & 1 & & 0 \\ & 2 & & 0 & \\ & & 1 & & \end{array}\right\}$};
   \node (14) at (0,-15) [draw,draw=none,scale=.7] {$\left\{\begin{array}{*{5}c} 3 & & 1 & & 0 \\ & 2 & & 0 & \\ & & 0 & & \end{array}\right\}$};
   \node (15) at (0,-18) [draw,draw=none,scale=.7] {$\left\{\begin{array}{*{5}c} 3 & & 1 & & 0 \\ & 1 & & 0 & \\ & & 0 & & \end{array}\right\}$};
   \node (4) at (-3,-6) [draw,draw=none,scale=.7] {$\left\{\begin{array}{*{5}c} 3 & & 1 & & 0 \\ & 3 & & 1 & \\ & & 1 & & \end{array}\right\}$};
   \node (7) at (-3,-9) [draw,draw=none,scale=.7] {$\left\{\begin{array}{*{5}c} 3 & & 1 & & 0 \\ & 2 & & 1 & \\ & & 1 & & \end{array}\right\}$};
   \node (10) at (-3,-12) [draw,draw=none,scale=.7] {$\left\{\begin{array}{*{5}c} 3 & & 1 & & 0 \\ & 1 & & 1 & \\ & & 1 & & \end{array}\right\}$};
   \node (13) at (-3,-15) [draw,draw=none,scale=.7] {$\left\{\begin{array}{*{5}c} 3 & & 1 & & 0 \\ & 1 & & 0 & \\ & & 1 & & \end{array}\right\}$};
   \node (3) at (3,-3) [draw,draw=none,scale=.7] {$\left\{\begin{array}{*{5}c} 3 & & 1 & & 0 \\ & 3 & & 0 & \\ & & 3 & & \end{array}\right\}$};
   \node (6) at (3,-6) [draw,draw=none,scale=.7] {$\left\{\begin{array}{*{5}c} 3 & & 1 & & 0 \\ & 3 & & 0 & \\ & & 2 & & \end{array}\right\}$};
   \node (9) at (3,-9) [draw,draw=none,scale=.7] {$\left\{\begin{array}{*{5}c} 3 & & 1 & & 0 \\ & 3 & & 0 & \\ & & 1 & & \end{array}\right\}$};
   \node (12) at (3,-12) [draw,draw=none,scale=.7] {$\left\{\begin{array}{*{5}c} 3 & & 1 & & 0 \\ & 3 & & 0 & \\ & & 0 & & \end{array}\right\}$};
   
   \draw [blue,->] (1) to["$1$"] (2);
   \draw [red,->] (1) to["$2$"] (3);
   \draw [blue,->] (2) to["$1$"] (4);
   \draw [red,->] (2) to["$2$"] (5);
   \draw [blue,->] (3) to["$1$"] (6);
   \draw [red,->] (4) to["$2$"] (7);
   \draw [blue,->] (5) to["$1$"] (7);
   \draw [red,->] (5) to["$2$"] (8);
   \draw [blue,->] (6) to["$1$"] (9);
   \draw [red,->] (7) to["$2$"] (10);
   \draw [blue,->] (8) to["$1$"] (11);
   \draw [red,->] (9) to["$2$"] (11);
   \draw [blue,->] (9) to["$1$"] (12);
   \draw [red,->] (10) to["$2$"] (13);
   \draw [blue,->] (11) to["$1$"] (14);
   \draw [red,->] (12) to["$2$"] (14);
   \draw [blue,->] (13) to["$1$"] (15);
   \draw [red,->] (14) to["$2$"] (15);
\end{tikzpicture}
\caption{Crystal graph of $\GTP(3,(3,1,0))$}
\end{subfigure}
\begin{subfigure}[b]{0.49\textwidth}
\begin{tikzpicture}[>=latex,line join=bevel,scale=.7]
   \node (1) at (0,0) [draw,draw=none,scale=.7] {$\begin{ytableau} 1 & 1 & 1 \\ 2  \end{ytableau}$};
   \node (2) at (0,-3) [draw,draw=none,scale=.7] {$\begin{ytableau} 1 & 1 & 2 \\ 2 \end{ytableau}$};
   \node (5) at (0,-6) [draw,draw=none,scale=.7] {$\begin{ytableau} 1 & 1 & 3 \\ 2 \end{ytableau}$};
   \node (8) at (0,-9) [draw,draw=none,scale=.7] {$\begin{ytableau} 1 & 1 & 3 \\ 3 \end{ytableau}$};
   \node (11) at (0,-12) [draw,draw=none,scale=.7] {$\begin{ytableau} 1 & 2 & 3 \\ 3 \end{ytableau}$};
   \node (14) at (0,-15) [draw,draw=none,scale=.7] {$\begin{ytableau} 2 & 2 & 3 \\ 3 \end{ytableau}$};
   \node (15) at (0,-18) [draw,draw=none,scale=.7] {$\begin{ytableau} 2 & 3 & 3 \\ 3 \end{ytableau}$};
   \node (4) at (-3,-6) [draw,draw=none,scale=.7] {$\begin{ytableau} 1 & 2 & 2 \\ 2 \end{ytableau}$};
   \node (7) at (-3,-9) [draw,draw=none,scale=.7] {$\begin{ytableau} 1 & 2 & 3 \\ 2 \end{ytableau}$};
   \node (10) at (-3,-12) [draw,draw=none,scale=.7] {$\begin{ytableau} 1 & 3 & 3 \\ 2 \end{ytableau}$};
   \node (13) at (-3,-15) [draw,draw=none,scale=.7] {$\begin{ytableau} 1 & 3 & 3 \\ 3 \end{ytableau}$};
   \node (3) at (3,-3) [draw,draw=none,scale=.7] {$\begin{ytableau} 1 & 1 & 1 \\ 3 \end{ytableau}$};
   \node (6) at (3,-6) [draw,draw=none,scale=.7] {$\begin{ytableau} 1 & 1 & 2 \\ 3 \end{ytableau}$};
   \node (9) at (3,-9) [draw,draw=none,scale=.7] {$\begin{ytableau} 1 & 2 & 2 \\ 3 \end{ytableau}$};
   \node (12) at (3,-12) [draw,draw=none,scale=.7] {$\begin{ytableau} 2 & 2 & 2 \\ 3 \end{ytableau}$};
   
   \draw [blue,->] (1) to["$1$"] (2);
   \draw [red,->] (1) to["$2$"] (3);
   \draw [blue,->] (2) to["$1$"] (4);
   \draw [red,->] (2) to["$2$"] (5);
   \draw [blue,->] (3) to["$1$"] (6);
   \draw [red,->] (4) to["$2$"] (7);
   \draw [blue,->] (5) to["$1$"] (7);
   \draw [red,->] (5) to["$2$"] (8);
   \draw [blue,->] (6) to["$1$"] (9);
   \draw [red,->] (7) to["$2$"] (10);
   \draw [blue,->] (8) to["$1$"] (11);
   \draw [red,->] (9) to["$2$"] (11);
   \draw [blue,->] (9) to["$1$"] (12);
   \draw [red,->] (10) to["$2$"] (13);
   \draw [blue,->] (11) to["$1$"] (14);
   \draw [red,->] (12) to["$2$"] (14);
   \draw [blue,->] (13) to["$1$"] (15);
   \draw [red,->] (14) to["$2$"] (15);
\end{tikzpicture}
\caption{Crystal graph of $\SSYT(3,(3,1,0))$}
\end{subfigure}
\caption{Isomorphic crystal graphs of type $A_2$}
\label{fig:crystals}
\end{figure}

Figure \ref{fig:crystals} shows the respective crystal graphs of two isomorphic crystals of type $A_2$. 
To illustrate, consider the Gelfand-Tsetlin pattern
\[
	\La = \left\{\begin{array}{*{5}c} \la^{(3)}_1 & & \la^{(3)}_2 & & \la^{(3)}_3 \\ & \la^{(2)}_1 & & \la^{(2)}_2 & \\ & & \la^{(1)}_1 & & \end{array}\right\}
= \left\{\begin{array}{*{5}c} 3 & & 1 & & 0 \\ & 3 & & 1 & \\ & & 2 & & \end{array}\right\}
\]
in the crystal $\GTP(3,(3,1,0))$ (in Figure \ref{fig:crystals} it is in the second vertex row from the top). Let us compute $\tilde f_1(\La)$. First we need $\varphi_1(\La)$.
The relevant diamond sum (see \ref{eq:diamond-sumA}) for $\La$ is (recall that entries outside the array are zero by convention)
\[
	A_1^{(1)}(\La) = a_1^{(1)}(\La) = \la_1^{(1)} - \la_1^{(0)} + \la_2^{(1)} - \la_2^{(2)} = 2 - 0 + 0 - 1 = 1,
\]
which gives
\[
	\varphi_1(\La) = \max\big\{ A_1^{(1)}(\La)\big\} = 1.
\]
By Definition \ref{def:GT-crystal-data},
the only value for $\ell$ here is $\ell=1$, hence applying $\tf_1$ on $\La$ has the effect of decrementing the entry $\la^{(1)}_1$:
\[
	\tilde{f}_1(\La) = \La-\Delta^{(1)}_1= \left\{\begin{array}{*{5}c} 3 & & 1 & & 0 \\ & 3 & & 1 & \\ & & 1 & & \end{array}\right\}
\]
as is visible in Figure \ref{fig:crystals}.
Let us also compute $\tf_2(\La)$. The diamond sums we need are
\begin{align*}
	A_1^{(2)}(\La) &= a_1^{(2)}(\La) + a_2^{(2)}(\La) = (3 - 2 + 1 - 1) + (1 - 0 + 0 - 0) = 1 + 1 = 2 \\
	A_2^{(2)}(\La) &= a_2^{(2)}(\La) = 1 - 0 + 0 - 0 = 1,
\end{align*}
which gives
\[
	\varphi_2(\La) = \max\big\{ A_1^{(2)}(\La), A_2^{(2)}(\La) \big\} = \max \big\{ 2, 1 \big\} = 2.
\]
Here, the largest index $\ell\in\{1,2\}$ for which $A^{(2)}_\ell=\varphi_2(\La)$ is $\ell=1$. Therefore,
applying $\tf_2$ on $\La$ has the effect of decrementing the entry $\la^{(2)}_1$:
\[
\tilde{f}_2(\La) = \La-\Delta^{(2)}_1= \left\{\begin{array}{*{5}c} 3 & & 1 & & 0 \\ & 2 & & 1 & \\ & & 2 & & \end{array}\right\}
\]
as can be seen in Figure \ref{fig:crystals}. The remaining crystal structure can be worked out in a similar fashion.

\end{document}